\documentclass[review]{elsarticle}

\usepackage{lineno,hyperref,amsmath, amsthm, enumerate,amssymb}
\modulolinenumbers[5]

\journal{Finite Fields and Their Applications}

\newtheorem{theorem}{Theorem}[section]
\newtheorem{lemma}[theorem]{Lemma}

\newtheorem{proposition}[theorem]{Proposition}
\newtheorem{corollary}[theorem]{Corollary}
\numberwithin{equation}{section}

\numberwithin{equation}{section}

\usepackage{tikz,float}
\usetikzlibrary{shapes.geometric}
\usetikzlibrary{shapes.arrows}
\usepackage{array}

\begin{document}

\begin{frontmatter}
	
	\title{Determinants of Matrices over Commutative Finite Principal Ideal Rings}
	\tnotetext[mytitlenote]{This work   is supported by the Thailand Research Fund  under Research Grant TRG5780065.}
	
	
	\author[mymainaddress]{Parinyawat Choosuwan}
	\ead{parinyawat.ch@gmail.com}
	
	\author[mysecondaryaddress]{Somphong Jitman \corref{mycorrespondingauthor}}
	\cortext[mycorrespondingauthor]{Corresponding author}
	\ead{sjitman@gmail.com}
	
	\author[mymainaddress]{Patanee Udomkavanich} 
	\ead{pattanee.u@chula.ac.th}
	
	\address[mymainaddress]{Department of Mathematics and Computer Science,
		Faculty of Science, \\ Chulalongkorn University,   Bangkok 10330,  Thailand}
	\address[mysecondaryaddress]{Department of Mathematics, Faculty of Science, 
		Silpakorn University, \\ Nakhon Pathom 73000,  Thailand}
	
	\begin{abstract}
 In this paper, the determinants of $n\times n$ matrices over  commutative finite chain rings and over  commutative  finite principal ideal rings  are studied.  The number of  $n\times n$ matrices over a commutative  finite  chain ring ${R}$  of a fixed determinant  $a$  is determined for all $a\in {R}$ and positive integers $n$.   Using the fact that every  commutative finite principal ideal ring is a product of commutative finite chain rings,     the  number   of  $n\times n$ matrices of a  fixed  determinant over a commutative finite principal ideal ring  is shown to be  multiplicative, and hence, it can be determined.  These results generalize the case of  matrices over the ring of  integers modulo $m$.
	\end{abstract}
	
	
\end{frontmatter}


\thispagestyle{empty}

\section{Introduction}

Determinants   are  known for their applications in matrix theory and linear algebra, e.g., 
determining the area of a triangle via Heron's formula  in \cite{K2004},  solving linear systems using Cramer's rule in \cite{CK2010}, and determining the singularity of a  matrix. 	 Therefore, properties of matrices and  determinants of matrices  have been extensively studied  (see \cite{CK2010}, \cite{MMMPSS2008}, and references therein).
Especially, matrices over finite fields are interesting due to their rich algebraic structures and various applications.   Singularity of such matrices is useful in applications.  For example, nonsingular matrices over finite fields  are good choices for constructing  good linear codes in \cite{BN2001}.  The number of  $n\times n$ singular (resp., nonsingular) matrices over a finite field $\mathbb{F}_q$ was studied in \cite{M1984}.    As a generalization of the prime   field $\mathbb{Z}_p$,  the determinants of matrices over the ring  $\mathbb{Z}_m$ of integers modulo $m$  were studied in \cite{BM1987} and \cite{LW2007}.   The number of $n\times n$ matrices over  $\mathbb{Z}_m$ of a fixed  determinant has been first studied in \cite{BM1987} .  In \cite{LW2007}, a different and simpler technique  was  applied to  determine the number of   such    matrices over  $\mathbb{Z}_m$.

Communicative finite principal ideal rings (CFPIRs), a generalization of the ring of integers modulo $m$,  are  interesting since they have applications  in many branches of mathematics and  links to other objects.   Cyclic codes of length $n$ over the finite field $\mathbb{F}_q$  are identified with the ideals in the  principal  ideal ring $\mathbb{F}_q[x]/\langle x^n-1\rangle$ (see  \cite{LX2004}).  The ring of $n\times n$ circulant matrices over  fields is a principal ideal ring (see \cite{KS2012}).  Some nonsingular matrices over a   CFPIR have been applied in constructing good matrix product codes in \cite{FLL2014}.  
Therefore, the determinants of matrices over CFPIRs are interesting.

To the best of our knowledge,  the enumeration of $n\times n$  matrices of a fixed determinant over  CFPIRs    has not been  completed.  It is therefore of natural interest to determine the number $d_n(\mathcal{R},r)$ of  $n\times n$  matrices of determinant $r$ over  a CFPIR $\mathcal{R}$.  Note that every   CFPIR $\mathcal{R}$ is a product of  commutative finite chain rings (CFCRs). This property allows us to separate the study into two steps:  1)   determine the number   $d_n(R,a)$ of $n\times n$ matrices over a  CFCR $R$ whose determinant is   $a $  for all $n\in \mathbb{N}$ and $a \in R$,   and 2) show that the number $d_n(\mathcal{R},r)$ is multiplicative among  the isomorphic components of $r$. The number $d_n(\mathcal{R},r)$ is therefore follows.

The paper is organized as follows. In Section $2$,  some  definitions and properties of rings and matrices are recalled. In   Section $3$,  the number $d_n(R,a)$ of $n\times n$ matrices over a  CFCR   $R$  having  determinant $a$   is determined for all $a \in R$ and  $n\in \mathbb{N}$.  In Section $4$, using the fact that every  CFPIR  is isomorphic to a product of CFCRs and results in Section $3$, the number  $d_n(\mathcal{R},r)$  of $n\times n$ matrices over a  CFPIR  $\mathcal{R}$  having  determinant $r$ is determined  for  all $r \in \mathcal{R}$ and  $n\in \mathbb{N}$.
\section{Preliminaries} 

In this section, definitions and some properties of rings and matrices are recalled.

A ring  $\mathcal{R}$  with identity $1\neq 0$ is called a
\textit{commutative  finite   principal ideal ring (CFPIR)}  if  $\mathcal{R} $ is  finite commutative and every ideal of  $R$ is principal.  A ring  $R$ is  called a
\textit{commutative finite chain ring (CFCR)} if it is finite commutative and its  ideals are linearly ordered by
inclusion. The properties of CFPIRs  and CFCRs can be found in \cite{GF2002},  \cite{H2001}, and \cite{HLM2003}. For completeness,  some properties used in this paper are recalled as follows.

From the definition of a CFCR, it is not difficult to see that every ideal in a   CFCR  $R$ is principal and $R$ has a
unique maximal ideal.  Let $\gamma$ be a generator of the  maximal ideal of $R$. Then the ideals in $R$ are of the form
\[R\supsetneq \gamma R \supsetneq \gamma^2 R\supsetneq \dots \supsetneq \gamma ^{e-1} R \supsetneq \gamma ^{e} R=\{0\}.\]
The smallest
positive integer $e$ such that $\gamma^e=0$ (or equivalently, $\gamma ^{e} R=\{0\})$  is called the \textit{nilpotency
    index} of $R$.  Since $\gamma R$ is  maximal in $R$, the  quotient ring $R/\gamma R$ is a finite  field and it is called the {\em residue field} of $R$. Both the characteristic and the cardinality
of a CFCR are  powers
of the characteristic of its residue field.  Denote by $U(R)$ the set of units in $R$. Then we have the following properties.

\begin{lemma}[{\cite{H2001} and \cite{HLM2003}}]  \label{lem:propCR}
    Let $R$ be a CFCR of nilpotency index $e$ and let $\gamma$ be  a generator of the maximal ideal of $R$.   Let $ V \subseteq R$  be  a set of representatives for
    the equivalence classes of $R$ under congruence modulo $\gamma$.
    Assume that the residue field $R/\langle \gamma\rangle$ is
    $\mathbb{F}_{q}$ for some prime power $q$. 
    Then the following statements hold.
    \begin{enumerate}[$1)$]
        \item For each $r\in R$,  there exist  unique  $a_0, a_1 ,\dots
        a_{e-1}\in
        V$ such that $$r=a_0+a_1\gamma+\dots+a_{e-1}\gamma^{e-1}. $$
        \item  $|V| = q$.
        \item $| \gamma^j R|=
        q^{e-j}$  for all $0\leq j\leq e$.
        \item $U(R)= \{a+\gamma b\mid a\in  V\setminus\{0\} \text{ and } b\in R\}$.
        \item $|U(R)|=(q-1)q^{e-1}$.
        \item For each  $0\leq i\leq e$, 
        $ R/\gamma^iR$ is a CFCR of nilpotency index $i$ and residue field $\mathbb{F}_q$.
    \end{enumerate}
\end{lemma}

\begin{proposition}[{\cite{GF2002}}] 
    \label{eqiv}
    Every CFPIR is a direct product of CFCRs.
\end{proposition}

%

Given a commutative ring $\mathfrak{R}$ and a positive integer $n$, let  $M_n(\mathfrak{R})$ denote the set of $n\times n$ matrices over the ring $\mathfrak{R}$.  Denote by  $GL_n(\mathfrak{R})$ the set of  invertible matrices in $M_n(\mathfrak{R})$. Equivalently,  $A\in GL_n(\mathfrak{R})$ if and only if $\det(A)$ is a unit in  $\mathfrak{R}$.

Denote by  $D_n(\mathfrak{R},a)$  the set of $n\times n$ matrices over $\mathfrak{R}$ whose determinant is $a$ and let $d_n(\mathfrak{R},a)=|D_n(\mathfrak{R},a)|$.  
The number   $d_n(\mathbb{F}_q,0)$ was studied in \cite{M1984} and extended to cover the number  $d_n(\mathbb{Z}_m,a)$   for all $a\in  \mathbb{Z}_m$ and for all positive integers $n$ in  \cite{BM1987}  and \cite{LW2007}. 
In  this paper, we focus on $d_n(\mathfrak{R},a)$ in the case where  $\mathfrak{R}$  is  CFCRs and   CFPIRs  which generalizes the results over $\mathbb{Z}_m$  in   \cite{BM1987}  and  \cite{LW2007}.

\section{Determinants of Matrices over Finite Chain Rings}

In  this section, we focus on the number $d_n(R,a)$ of $n\times n$ matrices  over  a CFCR $ R$. 

Let $R$ be  a CFCR of nilpotency index $e$ and residue field $\mathbb{F}_q$ and  let $\gamma$  be a generator  of the maximal ideal of $R$. For each $a\in R$,  by Lemma \ref{lem:propCR}, it is not difficult to see that $a =\gamma^sb$   for some $0\leq s\leq e$ and unit  $b\in U(R)$.  Precisely, $a=0$ if $s=e$, $a$ is a unit if $s=0$, and $a =\gamma^sb$ is a zero-divisor if $1\leq s\leq e-1$.

For each $a\in R$ and $n\in \mathbb{N}$, the formula of  $d_n(R,a)$  can be  determined using the above  three types of elements in $R$ summarized in the following diagram.

\begin{figure}[H]
    \centering
    \begin{tikzpicture} [
    block/.style    = { rectangle, draw=black, 
        fill=black!5, text width=9em, text centered,
        rounded corners, minimum height=1.8em },
    line/.style     = { draw, thick, ->, shorten >=2pt },
    ]
    \matrix [column sep=5mm, row sep=8  mm] {
        
        & \node [block] (oria) {$a \in R$ };            & \\ %
        
        & \node [block] (a) {$a=\gamma^s b$, $0\leq s \leq e$ and $b \in U(R)$ }; \\
        
        \node [block] (bri) {$a=\gamma^s b$, $0\leq s <e$   }; &  &\node [block] (a0) {$a =0$ };  \\ 
        
        \node [block] (aur1) {$d_n(R,a)=d_n(R,\gamma^s)$ };  && \\
        
        \node [block] (au1) {$a=1$ };  	   &\node [block] (rs) {$a=\gamma^s$ };  &	\node [block] (a01) {$d_n(R,0)$ };  \\
        
        \node [block] (au11) {$d_n(R,1) $ };  	   &\node [block] (rs1) {$ d_n(R,\gamma^s)$ };  &	  \\
        
        &\node [block] (end) {$ d_n(R,a)$ };  &	\\ 
    };
    \begin{scope} [every path/.style=line]
    \path (oria)        --   node{\quad\quad \quad\quad \quad \quad Lemma \ref{lem:propCR}}   (a); 
    
    \path (a)        --   node{$0\leq s <e$\quad\quad\quad \quad \quad  \quad \quad }   (bri); 
    \path (a)        --  node{\quad \quad \quad \quad $s=e$}  (a0);

    \path (bri)        --    node{\quad\quad\quad  \quad\quad\quad  Theorem \ref{ar=r}}(aur1);  
    
    \path (aur1)        --    node{\quad\quad\quad $e=0$}(au1);  
    \path (aur1)        --     node{\quad\quad\quad \quad\quad\quad \quad\quad $1 \leq s<e$}(rs);


    \path (a0)        --    node{\quad\quad\quad \quad\quad\quad  Theorem \ref{thma0}}(a01);  
    
    \path (au1)        --  node{\quad\quad\quad \quad\quad\quad  Theorem \ref{a=1}}  (au11);  
    \path (rs)        --   node{\quad\quad\quad \quad\quad\quad  Theorem \ref{a=r}} (rs1);

    \path (au11)        --    (end);  
    \path (rs1)        --    (end);  
    \path (a01)        --    (end);  
    \end{scope}
    %
    \end{tikzpicture}
    \caption{Steps in computing $d_n(R,a)$ over a CFCR  $R$}
\end{figure}
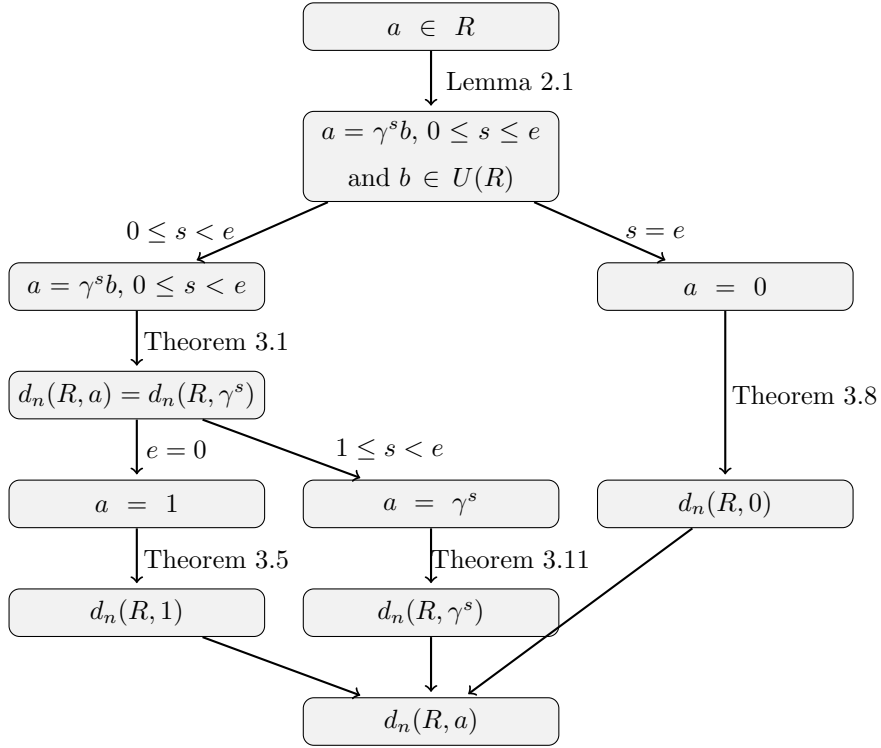

To simplify the computation, we give a relation between the number $d_n(R,\gamma^s)$ and  $d_n(R,\gamma^sb)$  for all units $b\in U(R)$ and integers $0\leq s\leq e$. 

\begin{theorem} \label{ar=r} Let $R$ be a CFCR of nilpotency index $e$ and residue field $\mathbb{F}_q$. If the maximal ideal of $R$ is generated by $\gamma$ and $0\leq s\leq e$, then  \[d_n(R,\gamma^s)=d_n(R,b\gamma^s)\] for all units $b$ in $U(R)$.
\end{theorem}
\begin{proof}  
    Let $b$ be a unit in $U(R)$ and let $0\leq s \leq e$ be an integer. If $s=e$, then $ \gamma^s=0=\gamma^sb$, and hence, we have $d_n(R,\gamma^s)= d_n(R,0)=d_n(R,b\gamma^s)$.  
    
    For each $0\leq s<e$, let $\alpha: D_n(R,\gamma^s)\to D_n(R,b\gamma^s)$  be a map defined by 
    \[\alpha(A)=\mathrm{diag}(b,1,\dots,1)A\]
    for all $A\in D_n(R,\gamma^s)$.   Note that,  for each  $A\in D_n(R,\gamma^s) $, $\det(A)= \gamma^s$ if  and only if $\det( \mathrm{diag}(b,1,\dots,1)A) =b\gamma^s$. It follows that $\alpha$ is well-defined. Since $b$ is a unit, the matrix $\mathrm{diag}(b,1,\dots,1)$ is invertible  which implies that $\alpha$ is a bijection.
    Therefore,  
    we have  \[d_n(R,\gamma^s)=|D_n(R,\gamma^s)|=|D_n(R,b\gamma^s)|=d_n(R,b\gamma^s)\]
    as desired.
\end{proof}

In the case where  $s=0$,  we have the following corollary. 
\begin{corollary} \label{a1} Let $R$ be a CFCR and let $n$ be a positive integer. Then  \[d_n(R,a)=d_n(R,1)\] for all units $a\in U(R)$.
\end{corollary}     

Next, the number $d_n(R,a)$ is determined  in three cases depending on the types of $a$, i.e., 1) $a$ is a unit, 2) $a$ is a zero-divisor, and 3) $a=0$. 

\subsection{The Number $d_n(R,a)$: $a$ is a Unit in $R$}
In this subsection,  we focus on   $d_n(R,a)$ in the case where $a$ is a unit in $U(R)$.  By Corollary \ref{a1}, it is suffices to determine only $d_n(R,1)$. 

First, the cardinality of $ GL_n(\mathbb{F}_q)$ which is key to determine the number $d_n(R,1)$ is recalled.
\begin{lemma}[{\cite{M1984}}]\label{GLnF} Let $q$ be a prime power and let $n$ be a positive integer. Then 
    \[|GL_n(\mathbb{F}_q)|=q^{n^2}\prod\limits _{i=1}^n (1-q^{-i}).\]
\end{lemma}

Next, the cardinality of  $GL_n(R)$ is determined. 
\begin{lemma} \label{GLnR} Let $R$ be a CFCR of nilpotency index $e$ and residue field $\mathbb{F}_q$ and let $n$ be a positive integer.  Then \[|GL_n(R)|=q^{en^2}\prod\limits _{i=1}^n (1-q^{-i}).\]
\end{lemma}
\begin{proof} In the case where $e=1$,  we have $R=\mathbb{F}_q$ and   \[|GL_n(R)|=|GL_n(\mathbb{F}_q)|=q^{n^2}\prod\limits_{i=1}^n (1-q^{-i})\]  by Lemma \ref{GLnF}.
    
    Assume that $e\geq 2$.
    Let $\gamma$ be a generator of the maximal ideal of $R$  and let  $\beta:M_n(R)\to M_n(R/ \gamma^{e-1}R)$ be a ring homomorphism defined by 
    \[\beta(A)=\overline{A},\]
    where  $\overline{[a_{ij}]}:=[a_{ij} +\gamma^{e-1}R]$ for all $[a_{ij}]\in M_n(R)$.   
    
    \noindent  Then $A\in \ker (\beta) $ if and only if the entries of $A$ are in $\gamma^{e-1}R$. By Lemma \ref{lem:propCR}, $|\ker (\beta)|= |\gamma^{e-1}R|^{n^2}=q^{n^2}$.  By the $1$st Isomorphism Theorem for rings, we have 
    \begin{align*}
        |M_n(R)|&=|\ker(\beta)||M_n(R/ \gamma^{e-1}R)|\\
        &=q^{n^2}|M_n(R/ \gamma^{e-1}R)|.
    \end{align*}
    
    For each $B\in M_n(R/ \gamma^{e-1}R)$, we have $\beta^{-1}(B)=\{A+\ker(\beta)\}$, where  $A\in M_n(R)$ is such that $\beta(A)=B$. 
    Note that   $A\in M_n(R)$ is invertible if and only if $\beta(A)$ is a unit  in $M_n(R/ \gamma^{e-1}R)$. It follows that, for each $B\in GL_n(R/ \gamma^{e-1}R)$, we  have \[\beta^{-1}(B)\subseteq GL_n(R) \text{ and }  |\beta^{-1}(B)|=|\ker(\beta)|.\]
    Hence,  \begin{align*}
        |GL_n(R)|&=|\ker(\beta)||GL_n(R/ \gamma^{e-1}R)|\\
        &=q^{n^2}|GL_n(R/ \gamma^{e-1}R)|.
    \end{align*}
    Continue this process, it can be concluded that 
    \begin{align*}
        |GL_n(R)|
        &=q^{n^2}|GL_n(R/ \gamma^{e-1}R)|\\
        &=q^{n^2}q^{n^2}|GL_n(R/ \gamma^{e-2}R)|\\
        &\ \ \vdots\\
        &=q^{(e-1)n^2} |GL_n(R/ \gamma R)|\\ 
        &=q^{(e-1)n^2} |GL_n(\mathbb{F}_q)|.
    \end{align*}
    By  Lemma \ref{GLnF},  we have  \[GL_n(\mathbb{F}_q)|=q^{n^2}\prod_{i=1}^n (1-q^{-i}),\] and hence, 
    \[|GL_n(R)| =q^{(e-1)n^2} |GL_n(\mathbb{F}_q)|= q^{en^2}\prod_{i=1}^n (1-q^{-i})\]
    as desired.
\end{proof}

The number of $n\times n$ matrices of determinant $1$ over a CFCR $R$ is now ready to determine in the next theorem. 

\begin{theorem}\label{a=1} Let $R$ be a CFCR of nilpotency index $e$ and residue field $\mathbb{F}_q$ and let $n$ be a positive integer. Then  \[d_n(R,1)=q^{e(n^2-1)}\prod_{i=2}^n (1-q^{-i}).\]  
\end{theorem}
\begin{proof}
    From the definition of $GL_n(R)$, it  follows  that $GL_n(R)$ is the  disjoint union of    $D_n(R,a)$ for all units $a\in U(R)$.  Precisely,
    \[GL_n(R)=\bigcup_{a\in U(R)} D_n(R,a)\]
    and $D_n(R,a)\cap D_n(R,b)=\emptyset$ for all $a\ne b$ in $U(R)$.
    
    By Corollary \ref{a1},  $d_n(R,1)=d_n(R,a)=|D_n(R,a)|$ is independent of $a$  for all units $a\in U(R)$. Hence, 
    \begin{align*} 
        |GL_n(R)|&= \sum_{a\in U(R)}  |D_n(R,a)|\\
        &=|U(R)|d_n(R,1).
    \end{align*}
    Therefore,  by Lemma \ref{lem:propCR} and Lemma \ref{GLnR}, it can be concluded that 
    \begin{align*} 
        d_n(R,1)&=\frac{|GL_n(R)|}  {|U(R)|}   \\
        &=\frac{q^{en^2}\prod\limits_{i=1}^n (1-q^{-i})}{(q-1)q^{e-1}}\\
        &=q^{e(n^2-1)}\prod_{i=2}^n (1-q^{-i})
    \end{align*}
    as desired.
\end{proof}

From Corollary \ref{a1} and Theorem \ref{a=1}, the next corollary follows immediately.
\begin{corollary}
    Let $R$ be a CFCR of nilpotency index $e$  and residue field $\mathbb{F}_q$ and let $n$ be a positive integer.  Then \[d_n(R,a)=d_n(R,1)=q^{e(n^2-1)}\prod_{i=2}^n (1-q^{-i})\] for all units $a$ in $R$.
\end{corollary}

\subsection{The Number $d_n(R,0)$} In this subsection, we focus the number $d_n(R,0)$.  Moreover, this number is key to determine  $d_n(R,a)$  in the case where $a$ is a zero divisor  in Subsection 3.3. 

First, we determine a relation among $d_n(R,0)$,  $d_{n-1}(R,0) $, and  $d_n(R/\gamma^{e-1}R,0+\gamma^{e-1}R)$. This relation plays an important role in determining the number $d_n(R,0)$ in Theorem \ref{thma0}.

\begin{lemma}\label{lema0}
    Let $R$ be a CFCR of nilpotency index $e$ and residue field $\mathbb{F}_q$ and let $n$ be a positive integer.  If $\gamma$  is a generator of the maximal ideal of $R$, then \begin{align*} d_n(R,0)&= \left(q^{en}-q^{(e-1)n} \right)q^{e(n-1)}d_{n-1}(R,0)  +q^{(n-1)n}d_n(R/\gamma^{e-1}R,0+\gamma^{e-1}R).\end{align*}
\end{lemma}
\begin{proof}
    Let $D^\prime_n(R,0) $ and  $D^{\prime\prime}_n(R,0)$ be sets defined to be \[D^\prime_n(R,0)=\{[a_{ij}]\in D_n(R,0)\mid \exists i \in \{1,2,\dots,n\} \text{ such that }  a_{i1}\in U(R) \}\] and  \[D^{\prime\prime}_n(R,0)=\{[a_{ij}]\in D_n(R,0)\mid  a_{i1}\notin U(R)  \text{ for all }  i \in \{1,2,\dots,n\}\}.\]
    Clearly, $D_n(R,0)=D^\prime_n(R,0)\cup D^{\prime\prime}_n(R,0)$ is a disjoint union.  It  therefore remains to show that  \[|D^\prime_n(R,0)|= \left(q^{en}-q^{(e-1)n} \right)q^{e(n-1)}d_{n-1}(R,0)\]  and \[| D^{\prime\prime}_n(R,0)|= q^{(n-1)n}d_n(R/\gamma^{e-1}R,0+\gamma^{e-1}R).\]
    
    Let $\rho:  D^\prime_n(R,0) \to D^\prime_n(R,0) $ be defined by
    \[A\to E, \]
    where 
    $E$ is obtained from $A$ by applying a sequence of elementary row operations such that $E_{11}=1$ and $E_{i1}=0$ for all $2\leq i\leq n$. It is not difficult to verify that $\rho$ is a $(q^{en}-q^{(e-1)n}) $-to-one function.
    
    Let $\nu: \rho(D^\prime_n(R,0) )\to D_{n-1}(R,0)$ be defined by 
    \[A\mapsto B, \]
    where $B$ is obtained by removing the first column and the first row of $A$.  Then  $\nu$ is a surjective   $q^{e(n-1)} $-to-one function. 
    
    Note that, for each $A\in D^\prime_n(R,0)$, we have  $\det(A)=0$ if and only if $\det(\rho(A))=0$, or equivalently,  $\det(\nu(\rho(A)))=0$.  It follows that $\nu\circ \rho$ is a     $\left(q^{en}-q^{(e-1)n} \right)q^{e(n-1)} $-to-one function from   $D^\prime_n(R,0)$ onto  $D_{n-1}(R,0)$, and hence,  \[|D^\prime_n(R,0)|= \left(q^{en}-q^{(e-1)n} \right)q^{e(n-1)}d_{n-1}(R,0).\]

    Next, we determine  the  cardinality of  $D^{\prime\prime}_n(R,0)$. Observe that,  for each $[a_{ij}]\in D^{\prime\prime}_n(R,0)$, we have  $a_{i1}\in \gamma R$   for all $1\leq i\leq n$.  By Lemma \ref{lem:propCR},  for each $1\leq i\leq n$,   we have $a_{i1}=\gamma b_i$ for some  $b_i\in \sum\limits_{j=0}^{e-2} \gamma^jV$ and $ V$  is defined in    Lemma \ref{lem:propCR}. 
    Let   $\psi: D^{\prime\prime}_n(R,0) \to M_n(R)$ be defined by 
    \[  [a_{ij}]\mapsto [b_{ij}],\] where  
    \[ b_{ij}=\begin{cases} b_i &\text{ if }j=1\\
    a_{ij} &\text{ if }j\ne 1.
    \end{cases}\]  Clearly,  $\psi $ is an injective map.

    Let  $\beta:M_n(R)\to M_n(R/ \gamma^{e-1}R)$ be a surjective ring homomorphism defined as in Lemma \ref{GLnF}  by 
    \[\beta(B)=\overline{B},\]
    where  $\overline{[b_{ij}]}:=[b_{ij} +\gamma^{e-1}R]$ for all $[b_{ij}]\in M_n(R)$.  
    
    For each $A\in D^{\prime\prime}_n(R,0) $, we have  $\det(A)=\gamma\det( \psi(A))$. Hence,   $\det(A)=0$  if and only if  $\det(\psi(A))\in \gamma^{e-1} R$, or equivalently,  
    \[\det(\beta(\psi(A)))=\det(\psi(A))  +\gamma^{e-1} R = 0 +\gamma^{e-1} R.\]  It follows that  $\beta\circ \psi$ is a surjective map, and hence, \[\beta(\psi (D^{\prime\prime}_n(R,0) )) = D_n(R/\gamma^{e-1} R,0+\gamma^{e-1} R) .\]   Observe that, for each  $C\in D_n(R/\gamma^{e-1} R,0+\gamma^{e-1} R) $,  there are exactly  $q^{(n-1)n}$ matrices  in  $\psi (D^{\prime\prime}_n(R,0) ) $ whose images under $\beta$ are $C$. Since $\psi$ is injective,  it follows that  $| D^{\prime\prime}_n(R,0)|= q^{(n-1)n}d_n(R/\gamma^{e-1}R,0+\gamma^{e-1}R)$.
\end{proof}

The number of $n\times n$ matrices of determinant $0$ over $R$  can be  determined as follows.
\begin{theorem} \label{thma0} Let $R$ be a CFCR of nilpotency index $e$ and residue field $\mathbb{F}_q$ and let $n$ be a positive integer.  Then \begin{align}\label{eqa0} d_n(R,0)=q^{en^2}\left(1-\prod_{i=0}^{n-1} (1-q^{-e-i})\right).\end{align}
\end{theorem}
\begin{proof} We prove the statement by   induction on $e$ and  $n$.  If $e=1$, then  $R=\mathbb{F}_q$ and \eqref{eqa0}  holds by Lemma \ref{GLnR}. If $n=1$, then  $d_n(R,0)=1$   which coincides with   \eqref{eqa0}.
    
    Assume that \eqref{eqa0} holds for all $k\in\{1,2,\dots, n-1\}$ and  $f\in\{1,2,\dots,e-1\}$.  Then  
    \begin{align*}
        d_k(R,0)&= \left(q^{fk}-q^{(f-1)k} \right)q^{f(k-1)}d_{k-1}(R,0)
        +q^{(k-1)k}d_k(R/\gamma^{f-1}R,0+\gamma^{f-1}R)  \\ & \quad \text{ by  Lemma \ref{lema0},}\\
        &= \left(q^{fk}-q^{(f-1)k} \right)q^{f(k-1)}q^{f(k-1)^2}\left(1-\prod_{i=0}^{k-2} (1-q^{-f-i})\right)\\
        &\quad +q^{(k-1)k}q^{(f-1)k^2}\left(1-\prod_{i=0}^{k-1} (1-q^{-f+1-i})\right)\\
        &\quad \text{ by the induction  hypothesis,}\\
        &=q^{fk^2}-\left( q^{fk^2}-q^{(f-1)k^2+(k-1)^2+(k-1)}\right)\prod_{i=0}^{k-2} (1-q^{-f-i}) \\
        &\quad -q^{(k-1)k+(f-1)k^2}\prod_{i=0}^{k-1} (1-q^{-f+1-i})\\
        &=q^{fk^2}-q^{fk^2}\prod_{i=0}^{k-2}(1-q^{f-i})\\
        &\quad +q^{(f-1)(k-1)^2+2(f-1)(k-1)+(k-1)^2+(k-1)}\prod_{i=0}^{k-2}(1-q^{-f-i})\\
        &=q^{fk^2}-q^{fk^2}(1-q^{-f-k+1})\prod_{i=0}^{k-2}(1-q^{-f-i})\\
        &=q^{fk^2}-q^{fk^2}\prod_{i=0}^{k-1}(1-q^{-f-i})\\
        &= 	q^{fk^2}\left(1-\prod_{i=0}^{k-1}(1-q^{-f-i})\right).
    \end{align*}
    Therefore, the result follows.
\end{proof}

\subsection{The Number $d_n(R,a)$: $a$ is a Zero-Divisor in $R$}  

In this subsection, we focus on  $d_n(R,a) $  in the case where $a$ is a zero-divisor.  In this case, $a=\gamma^sb$ for some $1\leq s<e$ and $b\in U(R)$. From Theorem \ref{ar=r}, it suffices to determine only the number  $d_n(R,\gamma^s) $  for all $1\leq s<e$. 

The following preliminary results are key to determine the number  $d_n(R,\gamma^s) $.

\begin{lemma}\label{lems} Let $R$ be a CFCR of nilpotency index $e\geq 3$ and residue field $\mathbb{F}_q$ and let $n$ be a positive integer.
    If  $\gamma$ is a generator of the maximal ideal of $R$, then \[d_n(R,\gamma^s)=q^{(n^2-1)}d_n(R/\gamma^{e-1}R,\gamma^s+ \gamma^{e-1}R)\] for all $1\leq s<e-1$.
\end{lemma}
\begin{proof}
    Let $1\leq s<e-1$ be an integer and let   $\beta:M_n(R)\to M_n(R/ \gamma^{e-1}R)$ be a ring homomorphism defined as in Lemma \ref{GLnF}  by 
    \[\beta(A)=\overline{A},\]
    where  $\overline{[a_{ij}]}:=[a_{ij} +\gamma^{e-1}R]$ for all $[a_{ij}]\in M_n(R)$.
    Note that, 
    for each $A\in M_n(R)$,   $\det(\beta(A))=\gamma^s+ \gamma^{e-1}R$ if and only if $\det (A)= \gamma^s+ \gamma^{e-1}b$  for some  $b\in V$, where $V$ is defined in Lemma \ref{lem:propCR}.
    Since $1\leq e-s-1< e-1$, it follows that  $1+\gamma^{e-s-1}b$ is a unit in $U(R)$. Hence,  
    \begin{align*}
        |\{A\in M_n(R)\mid  &\det(A) = \gamma^s+ \gamma^{e-1}b\}|\\
        &= |\{A\in M_n(R)\mid \det(A) = \gamma^s(1+ \gamma^{e-s-1}b)\}|\\
        &= |\{A\in M_n(R)\mid \det(A) = \gamma^s \}|\\
        &=d_n(R, \gamma^s).
    \end{align*}	  
    Equivalently, 
    \begin{align}\label{eq11}
        |\{A\in M_n(R)\mid \det(\beta(A)) = \gamma^s+ \gamma^{e-1}R \}|
        &= |V|d_n(R, \gamma^s)= qd_n(R, \gamma^s).
    \end{align}	  
    
    As in  the proof of Lemma \ref{GLnF}, we have  $|\ker (\beta)| =q^{n^2}$. Hence,
    \begin{align}\label{eq22}
        |\{A\in M_n(R)\mid &\det(\beta(A)) = \gamma^s+ \gamma^{e-1}R  \}|\notag \\
        &= q^{n^2}|\{B\in M_n(R/\gamma^{e-1}R)\mid \det(B) = \gamma^s +\gamma^{e-1}R\}|\notag\\
        &= q^{n^2}d_n(R/\gamma^{e-1}R, \gamma^s+\gamma^{e-1}R).
    \end{align}	  
    Combining \eqref{eq11} and \eqref{eq22},  it can be concluded that 
    \[ qd_n(R, \gamma^s) = q^{n^2}d_n(R/\gamma^{e-1}R, \gamma^s+\gamma^{e-1}R).\] 
    Therefore,  \[d_n(R,\gamma^s)=q^{(n^2-1)}d_n(R/\gamma^{e-1}R,\gamma^s+ \gamma^{e-1}R)\] as desired.
\end{proof}
Applying Lemma \ref{lems} recursively,  the next  corollary follows.
\begin{corollary}\label{corlast} Let $R$ be a CFCR of nilpotency index $e+f$ and residue field $\mathbb{F}_q$, where $2\leq e$ and $1\leq f$ are integers.
    If the maximal ideal of $R$ is generated by $\gamma$, then \[d_n(R,\gamma^s)=q^{f(n^2-1)}d_n(R/\gamma^eR,\gamma^s+ \gamma^eR)\] for all $1\leq s<e$.
\end{corollary}

Now, we are ready to determined  the number  $d_n(R,\gamma^s) $ of $n\times n$ matrices over a CFCR $R$ whose determinant is $\gamma^s$.

\begin{theorem}\label{a=r} Let $R$ be a CFCR of nilpotency index $e$ and residue field $\mathbb{F}_q$ and let $n$ be a positive integer.  If the maximal ideal of $R$ is generated by $\gamma$,  then \[d_n(R,\gamma^s)=\frac{q^n-1}{q-1}q^{en^2-n-e+1}\prod_{i=1}^{n-1} (1-q^{-s-i})\]
    for all  integers $1\leq s<e$.
\end{theorem}
\begin{proof}
    Let  $1\leq s<e$ be an integer and let  $\mu:M_n(R/ \gamma^{s+1}R)\to M_n(R/ \gamma^{s}R)$ be a ring homomorphism defined by 
    \[\mu(A)=\overline{A},\]
    where  $\overline{[a_{ij} +\gamma^{s+1}R]}:=[a_{ij} +\gamma^{s}R]$ for all $[a_{ij}+\gamma^{s+1}R]\in M_n(R/ \gamma^{s+1}R)$.  Then,  for each $A\in M_n(R/ \gamma^{s+1}R)$, $\det(\mu(A))=0+\gamma^{s}R $ if and only if $\det(A)=\gamma^{s} b+\gamma^{s+1}R$ for some $b\in V$, where $V$ is defined in Lemma \ref{lem:propCR}.  Since $|\ker(\mu)|=q^{n^2}$,  we have 
    \begin{align*}
        q^{n^2} d_n(R/\gamma^{s}R,0+\gamma^{s}R)
        &= |\ker(\mu)|d_n(R/\gamma^{s}R,0+\gamma^{s}R)\\&=|M_n(R/ \gamma^{s+1}R)|\\
        &=d_n(R/\gamma^{s+1}R,0+\gamma^{s+1}R)\\ &\quad +\sum_{b\in V\setminus \{0\}} d_n(R/\gamma^{s+1}R,\gamma^{s}b+\gamma^{s+1}R)\\
        &=d_n(R/\gamma^{s+1}R,0+\gamma^{s+1}R)\\
        & \quad +(q-1)	d_n(R/\gamma^{s+1}R,\gamma^{s}+\gamma^{s+1}R)
    \end{align*}
    by  Theorem \ref{ar=r}.    Hence,  we have 
    \begin{align} \label{eq1}
        d_n(R/\gamma^{s+1}&R,\gamma^{s}+\gamma^{s+1}R)\notag  \\
        &= \frac{1}{q-1}\left(q^{n^2} d_n(R/\gamma^{s}R,0+\gamma^{s}R)-d_n(R/\gamma^{s+1}R,0+\gamma^{s+1}R)\right).
    \end{align}
    By Corollary \ref{corlast},   we have 
    \begin{align} \label{eq2}
        d_n(R,\gamma^{s}) &= 	d_n(R/\gamma^{e+1+(s-e-1)}R,\gamma^{s}+\gamma^{e+1+(s-e-1)}R) \notag \\
        &=q^{(e-s-1)(n^2-1)}  d_n(R/\gamma^{s+1}R,\gamma^{s}+\gamma^{s+1}R).
    \end{align}
    Combining (\ref{eq1}) and (\ref{eq2}),  it follows that 
    \begin{align*}  
        d_n(R,\gamma^{s})  
        &=\frac{q^{(e-s-1)(n^2-1)} }{q-1} \Bigg(q^{n^2} d_n(R/\gamma^{s}R,0+\gamma^{s}R)\\
        &\quad \hskip10em   -d_n(R/\gamma^{s+1}R,0+\gamma^{s+1}R)\Bigg).
    \end{align*}
    Applying Theorem \ref{thma0}, we have 
    \begin{align*}  
        d_n(R,\gamma^{s})  
        &=\frac{q^{(e-s-1)(n^2-1)} }{q-1} \left(q^{n^2} q^{sn^2}\left(1-\prod_{i=0}^{n-1} (1-q^{-s-i})\right) \right.\\
        &\hskip10em\left.-q^{(s+1)n^2}\left(1-\prod_{i=0}^{n-1} (1-q^{-s-1-i})\right)\right)\\
        &=\frac{q^{n^2e-e+s+1}}{q-1} 
        \left( -\prod_{i=0}^{n-1} (1-q^{-s-i})
        +\prod_{i=0}^{n-1} (1-q^{-s-1-i})\right) \\
        &=\frac{q^{n^2e-e+s+1}}{q-1} \prod_{i=1}^{n-1} (1-q^{-s-i})\left( (1-q^{-s-n})-(1-q^{-s})\right)\\
        &=\frac{q^{n^2e-e+1}}{q-1}(1-q^{-n})\prod_{i=1}^{n-1} (1-q^{-s-i})\\
        &=\frac{q^n-1}{q-1}q^{en^2-n-e+1}\prod_{i=1}^{n-1} (1-q^{-s-i})
    \end{align*}
    as desired.
\end{proof}

%

\section{Determinants of Matrices over Finite Principal Ideal Rings}

In this section, we focus on a more general case. The number of $n\times n$ matrices of a fixed determinant  over   CFPIRs is determined.  

Let $\mathcal{R}$ be a CFPIR. With out loss of generality, by Proposition \ref{eqiv},  it can be assume that  $\mathcal{R}=R_1\times R_2\times \dots\times R_m$ for some positive  integer $m$, where $R_i$ is a CFCR for all $1\leq i\leq m$.  For each $1\leq i\leq m$, let $\phi_i:\mathcal{R}\to R_i$ be a projection map defined by \[\phi_i((r_1,r_2,\dots,r_m))=r_i.\]  Note that  $\phi_i$ is a surjective ring homomorphism for all $1\leq i\leq m$.


The number of $n\times n$ matrices of a fixed determinant  over   $\mathcal{R}$ can be determined as follows.
\begin{theorem} Let  $\mathcal{R}=R_1\times R_2\times R_m$ be a CFPIR where $R_1, R_2, \dots, R_m$  be  CFCRs and let $n$ be a positive integer. Let $r\in \mathcal{R}$ and let  $\phi_i$'s  be defined as above.  Then \[ d_n(\mathcal{R},r)=d_n(R_1,\phi_1(r))d_n(R_2,\phi_2(r)),\dots, d_n(R_m,\phi_m(r) ).\]
\end{theorem}

\begin{proof}
    It is sufficient  to prove only  the case $\mathcal{R}=R_1\times R_2$. The rest can be obtained by  induction on $m$. 
    
    From the definition of $d_n(\mathcal{R},r)$, we show that 
    \[|D_n(\mathcal{R},r)|=|D_n(R_1,\phi_1(r))||D_n(R_2,\phi_2(r))|.\]  
    
    Let $\Phi: M_n(\mathcal{R}) \to  M_n( R_1)\times M_n(R_2)$ be a ring isomorphism defined by 
    \[[a_{ij}]\mapsto  ([\phi_1(a_{ij})], [\phi_2(a_{ij})]) . \]  
    Since   $\Phi$ is injective, it suffices to show that the isomorphism  $\Phi$ maps  $D_n(\mathcal{R},r)$ onto  $D_n(R_1,\phi_1(r))\times D_n(R_2,\phi_2(r))$.   Since $r=(\phi_1(r),\phi_2(r) ) $, we have  \[\Phi(D_n(\mathcal{R},r))\subseteq D_n(R_1,\phi_1(r))\times D_n(R_2,\phi_2(r)).\]  Let $(B_1,B_2)\in  D_n(R_1,\phi_1(r))\times D_n(R_2,\phi_2(r))$.  Since $\Phi$ is surjective, there exists $B\in M_n(\mathcal{R})$ such that $\Phi(B)=(B_1,B_2)$ and  $\det(B_1)=\phi_1(r)$ and  $\det(B_2)=\phi_2(r)$.  As  $r=(\phi_1(r),\phi_2(r) ) $, it follows that  $\det(B)=r$, and hence,   $A\in D_n(\mathcal{R},r)$.   Therefore, $|D_n(\mathcal{R},r)|=|D_n(R_1,\phi_1(r))||D_n(R_2,\phi_2(r))|$ as desired.
\end{proof}

\section{Conclusion Remarks}
Determinants of matrices over CFCRs (resp.,  CFPIRs) $R$  are studied. For a given positive integer $n$ and $a\in R$,    the number of  $n\times n$ matrices of determinant $a$ over  ${R}$  is determined.  These generalize the results  on  the determinants of matrices over  $\mathbb{Z}_m$  in \cite{LW2007}.   This counting problem over commutative  finite local rings or over arbitrary commutative finite rings would be also  interesting.

\section*{Acknowledgements}
The authors wold like to  thank    the anonymous referees for there helpful comments. This work   is supported by the Thailand Research Fund  under Research Grant TRG5780065.


\end{document}